\documentclass[10pt]{amsart}
\usepackage[leqno]{amsmath}
\usepackage{amssymb,latexsym,soul,cite,amsthm,color,enumitem,graphicx,mathtools,microtype,accents}
\usepackage[colorlinks=true,urlcolor=green,citecolor=green,linkcolor=green,linktocpage,pdfpagelabels,bookmarksnumbered,bookmarksopen]{hyperref}
\definecolor{green}{rgb}{0.13, 0.55, 0.13}
\usepackage[english]{babel}
\usepackage[left=2.5cm,right=2.5cm,top=2.5cm,bottom=2.5cm]{geometry}

\numberwithin{equation}{section}

\newtheorem{theorem}{Theorem}[section]
\theoremstyle{plain}
\newtheorem{lemma}[theorem]{Lemma}
\theoremstyle{plain}
\newtheorem{proposition}[theorem]{Proposition}
\theoremstyle{plain}

\theoremstyle{definition}
\newtheorem{remark}[theorem]{Remark}

\newcommand{\N}{{\mathbb N}}

\newcommand{\R}{{\mathbb R}}
\newcommand{\eps}{\varepsilon}
\newcommand{\beq}{\begin{equation}}
\newcommand{\eeq}{\end{equation}}
\renewcommand{\le}{\leqslant}
\renewcommand{\ge}{\geqslant}

\makeatletter
\newcommand{\leqnomode}{\tagsleft@true}
\newcommand{\reqnomode}{\tagsleft@false}
\makeatother

\newenvironment{enumroman}{\begin{enumerate}

}{\end{enumerate}}

\title[Regularity for fractional $p$-Laplacian]{H\"older regularity for the fractional $p$-Laplacian, revisited}

\author[F.M.\ Cassanello, F.G.\ D\"uzg\"un, A.\ Iannizzotto]{Filippo Maria Cassanello, Fatma Gamze D\"uzg\"un, Antonio Iannizzotto}

\address[]{Dipartimento di Matematica e Informatica
\newline\indent
Universit\`a di Cagliari
\newline\indent
Via Ospedale 72, 09124 Cagliari, Italy}
\email{filippom.cassanello@unica.it}
\email{fatmagamze.duzgun@unica.it}
\email{antonio.iannizzotto@unica.it}


\begin{document}

\begin{abstract}
We present an alternative proof for local H\"older regularity of the solutions of the fractional $p$-Laplace equations, based on clustering and expansion (more precisely, recentering) of positivity.
\end{abstract}

\maketitle

\begin{center}
Version of \today\
\end{center}

\section{Introduction and main result}\label{sec1}

\noindent
This paper is devoted to the study of regularity properties of the weak solutions of the following class of nonlocal, nonlinear equations:
\beq\label{fpl}
\mathcal{L}_K u = 0 \ \text{in $\Omega$.}
\eeq
Here $\Omega\subseteq\R^N$ ($N\ge 2$) is an open connected set, bounded or unbounded. The involved operator is heuristically defined as
\[\mathcal{L}_K u(x) = \lim_{\eps\to 0^+}\int_{B_\eps^c(x)}|u(x)-u(y)|^{p-2}(u(x)-u(y))K(x,y)\,dy,\]
where $p>1$, $s\in(0,1)$ are real s.t.\ $ps<N$, and the kernel $K:\R^N\times\R^N\to\R$ is a measurable function s.t.\ for a.e.\ $(x,y)\in\R^N\times\R^N$
\begin{itemize}[leftmargin=1cm]
\item[$(K_1)$] $K(x,y)=K(y,x)$;
\item[$(K_2)$] $\Lambda_1\le K(x,y)|x-y|^{N+ps}\le \Lambda_2$ ($0<\Lambda_1\le \Lambda_2$).
\end{itemize}
The prototype of this family of operators is the $s$-fractional $p$-Laplacian, which is obtained setting $\Lambda_1=\Lambda_2$ in $(K_2)$ and is hence represented (at least for $p\ge 2$ and very regular $u$) as
\[(-\Delta)_p^s u(x) = C_{N,p,s}\lim_{\eps\to 0^+}\int_{B_\eps^c(x)}\frac{|u(x)-u(y)|^{p-2}(u(x)-u(y))}{|x-y|^{N+ps}}\,dy,\]
where $C_{N,p,s}>0$ is a suitable normalizing constant. In particular, for $p=2$ this operator becomes the (linear) fractional Laplacian, while (via a convenient definition of $C_{N,p,s}$) for $s\to 1$ we retrieve the negative $p$-Laplacian as a limit case. For precise definitions of the operator and solutions of equation \eqref{fpl}, see Section \ref{sec2} below.
\vskip2pt
\noindent
Regularity theory for fractional order equations is strongly influenced by the nonlocal behavior of the involved operators, i.e., by the fact that the values of $\mathcal{L}_K u$ in a ball depend not only on the behavior of $u$ in a larger ball, but as well on its behavior at infinity. In the linear case $p=2$, this difficulty can be overcome by considering the solution $u$ as the trace of a solution of a convenient local elliptic equation, via the celebrated extension result of \cite{CS}. A direct approach is also possible, see for instance \cite{FKV} for the elliptic case (with a wide family of kernels) and \cite{FK} for the parabolic case. 
\vskip2pt
\noindent
In the nonlinear case, no analogous extension result is known so far, therefore nonlocality must be faced directly. The first result in this direction goes back to \cite{DCKP}, where it is proved that solutions of \eqref{fpl} with Dirichlet type conditions are locally H\"older continuous. In the same paper, the fundamental notion of {\em nonlocal tail} is explicitly introduced: for all $x_0\in\R^N$, $R>0$ we set
\beq\label{tail}
{\rm Tail}(u,x_0,R) = \Big[R^{ps}\int_{B_R^c(x_0)}\frac{|u(x)|^{p-1}}{|x-x_0|^{N+ps}}\,dx\Big]^\frac{1}{p-1}.
\eeq
The approach of \cite{DCKP} is based on a logarithmic estimate, a Caccioppoli type inequality involving tail terms, and a De Giorgi-Nash-Moser type iterative argument. A similar method is employed to prove Harnack inequalities for solutions of \eqref{fpl} \cite{DCKP1}. Such results have been extended and improved in several works. We just mention some recent contributions: for the degenerate elliptic case $p>2$, higher H\"older regularity has been achieved in \cite{BLS} via higher differentiability \cite{BL}, and in \cite{BDLMBS} via a direct finite-difference scheme, up to a quasi-Lipschitz regularity for $s$ big enough (for the classical kernel); the singular case $1<p<2$ is considered in \cite{GL}, reaching higher H\"older regularity, and in \cite{DKLN} studying higher differentiability (most results allow for a non-zero reaction term). On a different path, regularity for functions in fractional De Giorgi classes has been studied in \cite{C} (see also \cite{CCMV}). For the parabolic equation, H\"older regularity is proved in \cite{BLS1,L1} (see also \cite{DZZ}). The reader can find exhaustive information in the survey papers \cite{KMS,MS,P}. Note that, in general, local H\"older regularity is a meaningful property for solutions of nonlocal equations, as it represents the {\em optimal} type of regularity one can expect to keep up to the boundary (for boundary regularity, see \cite{ROS} for the linear case and \cite{I,IM} for the nonlinear one).
\vskip2pt
\noindent
\vskip2pt
\noindent
Here we propose an alternative approach, proposed in \cite{DMV} for elliptic equations of the $p$-Laplace type and developed in \cite{DMV1} for anisotropic operators (which in turn is based on the ideas of \cite{DB} for parabolic equations). The result of \cite{DMV} deals with a general nonlinear elliptic equation, which for simplicity we can identify with the (local) $p$-Laplace equation
\beq\label{pl}
\Delta_p u = 0 \ \text{in $\Omega$.}
\eeq
The result of \cite{DMV} ensures H\"older continuity of solutions of \eqref{pl}, using a local clustering lemma from \cite{DBGV} for functions in Sobolev spaces: {\em local clustering} means that any function in $W^{1,p}(B_R)$ (here $B_R$ is a ball in $\R^N$ with radius $R>0$), subject to a bound on the seminorm $\|\nabla u\|_{L^p(B_R)}$, and s.t.\ $u$ is greater than a given level $k$ in a subset of $B_R$ with a fixed measure quotient, is in fact greater than any small multiple of $k$ in an arbitrary large subset of a ball $B_{\eps R}(x_1)$ centered at a point $x_1\in B_R$, having as radius a conveniently small multiple of $R$.
\vskip2pt
\noindent
By means of local clustering and a De Giorgi type iterative lemma, it can be seen that any locally bounded weak solution $u$ of \eqref{pl}, with essential infimum $\mu$ and oscillation $\omega$ in a ball $B_R(x_0)\subset\Omega$, is in fact greater than $\mu$ plus a multiple of $\omega$ in a reduced and translated ball $B_{\eps R}(x_1)\subset B_R(x_0)$ ({\em positivity clustering}). The second step consists in proving that $u$ is greater than $\mu$ plus a (even smaller) multiple of $\omega$ in a thin cylinder centered at $x_1$, along an arbitrary direction, so to reach finally a reduced ball centered again at $x_0$ ({\em positivity expansion}). The latter property is then employed to achieve an oscillation estimate, which implies H\"older continuity through a standard argument.
\vskip2pt
\noindent
To our knowledge, this approach has not been tried on equation \eqref{fpl} yet, though an independent work from us applies a similar argument to a linear fractional equation with measure-type kernel \cite{C1}. Thus, the purpose of the present work is to use this method to find an alternative proof of H\"older continuity. Some of the required tools (logarithmic estimate, Caccioppoli inequality) for equation \eqref{fpl} can be picked from \cite{DCKP}, while a local clustering lemma in fractional Sobolev spaces has been recently obtained in \cite{DIV}, with the gradient seminorm replaced by the Gagliardo seminorm in a ball. A suitable De Giorgi type lemma is proved below (see Section \ref{sec3}). Thus, the main difficulty lies in adapting the clustering/expansion method to the nonlocal case.
\vskip2pt
\noindent
For simplicity, let us keep the notation above. First, the quantity $\omega$ cannot be simply defined as the oscillation of $u$ in $B_R(x_0)$, but it should rather contain both the $L^\infty$-norm of $u$ and a tail term defined as in \eqref{tail} (we borrow this idea from the parabolic case treated in \cite{L1}). Positivity clustering can be obtained by suitably adapting the known arguments. Instead, the peculiar nature of the operator does not seem to fit with the geometry of positivity expansion, while it allows us to directly shift the lower bound on $u$ to a reduced ball centered at $x_0$ ({\em positivity recentering}), thus leading to an oscillation estimate between two concentric balls (Section \ref{sec4}). Further, such estimate must be developed into an iterative one by strong induction, considering a nested sequence of alternatives (Section \ref{sec5}). Finally, H\"older continuity is deduced as usual (Section \ref{sec6}).
\vskip2pt
\noindent
The outcome of our labor is the following result:

\begin{theorem}\label{hol}
Let $\Omega\subset\R^N$ be an open connected set, $p>1$, $s\in(0,1)$ be real s.t.\ $ps<N$, $K$ satisfy $(K_1)-(K_2)$, and $u$ be a locally bounded, local weak solution of \eqref{fpl}. Then, $u\in C^{0,\alpha}_{\rm loc}(\Omega)$ with $\alpha\in(0,1)$ depending on $N,p,s,\Omega,\Lambda_1,\Lambda_2$. Moreover, for all $\Omega'\Subset\Omega$ s.t.\ ${\rm dist}(\Omega',\partial\Omega)=2R>0$, and all $x_1,x_2\in\Omega'$ we have
\[|u(x_1)-u(x_2)| \le C\sup_{x\in\Omega'_R}\big[|u(x)|+{\rm Tail}(u,x,R)\big]\Big(\frac{|x_1-x_2|}{R}\Big)^\alpha,\]
where $C>0$ depends on $N,p,s,\Omega,\Lambda_1,\Lambda_2$, and $\Omega'_R$ is the $R$-dilatation of $\Omega'$
\[\Omega'_R = \big\{x\in\Omega:\,{\rm dist}(x,\Omega')\le R\big\}.\]
\end{theorem}

\noindent
Our result is similar to that of \cite{DCKP}, except that we do not impose Dirichlet-type conditions. Also, it is closely related other results mentioned above (see for instance \cite{BLS,GL}). The novelty of our work lies mainly in the proof, which we believe to be more geometric and intuitive than in the previous literature and, despite some delicate passages, seems to be a natural approach to regularity. Also, we remark that our result works for all $p>1$ and a general kernel, with no difference in proofs, and can be extended to functions in fractional De Giorgi classes, see \cite{CCMV}. Finally, it is surprising that the approach of \cite{DMV} also works in the nonlocal case, since it involves properties (like positivity clustering) which are typically {\em local}. Also, we hope that this alternative approach may be useful towards a better understanding of even more involved operators (for instance, with inhomogenous or anisotropic fractional operators).
\vskip4pt
\noindent
{\bf Notation.} Throughout the paper, for all $U\subset\R^N$ we will denote by $|U|$ the $N$-dimensional Lebesgue measure of $U$, $U^c=\R^N\setminus U$. By $B_R(x)$ (resp., $\overline{B}_R(x_0)$) we will denote the open (resp., closed) ball of radius $R$ centered at $x$. Writings like $u\le v$ in $U$ will mean that $u(x)\le v(x)$ for a.e.\ $x\in U$. By $u_+$ (resp., $u_-$) we will denote the positive (resp., negative) part of $u$. By $\inf_U u$ (resp., $\sup_U u$) we will denote the essential infimum (resp., supremum) of $u$ in $U$, and we will define the oscillation as ${\rm osc}_U u =\sup_U u-\inf_U u$. Most important, $C$ will denote several positive constants, only depending on the data $N,p,s,\Omega,\Lambda_1,\Lambda_2$ of the problem.

\section{Preliminaries}\label{sec2}

\noindent
We begin with a quick recall of some basic definitions about fractional Sobolev spaces and the functional structure behind equation \eqref{fpl}, referring the reader to \cite{DNPV,L} for a detailed account on the subject. For simplicity, we denote all domains by $\Omega$. First, for any measurable function $u:\Omega\to\R$ we define the Gagliardo seminorm
\[[u]_{s,p,\Omega} = \Big[\iint_{\Omega\times\Omega}\frac{|u(x)-u(y)|^p}{|x-y|^{N+ps}}\,dx\,dy\Big]^\frac{1}{p}.\]
Accordingly we define the fractional Sobolev space
\[W^{s,p}(\Omega) = \big\{u\in L^p(\Omega):\,[u]_{s,p,\Omega}<\infty\big\},\]
endowed with the norm
\[\|u\|_{W^{s,p}(\Omega)} = \big([u]_{s,p,\Omega}^p+\|u\|_{L^p(\Omega)}^p\big)^\frac{1}{p}.\]
If $\Omega$ has a smooth enough boundary, then $W^{s,p}(\Omega)$ is continuously embedded into $L^{p^*_s}(\Omega)$, with $p^*_s=Np/(N-ps)$ (see \cite[Theorem 6.7]{DNPV}). If $\Omega$ is bounded, then we set
\[\widetilde{W}^{s,p}(\Omega) = \Big\{u\in L^p_{\rm loc}(\R^N):\,u\in W^{s,p}(\Omega') \ \text{for some $\Omega'\Supset\Omega$,} \ \int_{\R^N}\frac{|u(x)|^{p-1}}{(1+|x|)^{N+ps}}\,dx < \infty\Big\}.\]
Finally, for a general $\Omega$ we set
\[\widetilde{W}^{s,p}_{\rm loc}(\Omega) = \big\{u\in L^p_{\rm loc}(\R^N):\,u\in\widetilde{W}^{s,p}(\Omega') \ \text{for all bounded $\Omega'\subseteq\Omega$}\big\}.\]
Clearly, for all $u\in\widetilde{W}^{s,p}_{\rm loc}(\Omega)$ and all $x_0\in\Omega$, $R>0$ the tail defined in \eqref{tail} is finite.
\vskip2pt
\noindent
We say that $u\in\widetilde{W}^{s,p}_{\rm loc}(\Omega)$ is a (weak local) supersolution of \eqref{fpl} if for all $\varphi\in C^1_c(\Omega)$, $\varphi\ge 0$ in $\Omega$ we have
\[\iint_{\R^N\times\R^N}|u(x)-u(y)|^{p-2}(u(x)-u(y))(\varphi(x)-\varphi(y))K(x,y)\,dx\,dy \ge 0.\]
The definition of a subsolution is analogous. Finally, $u$ is a solution of \eqref{fpl} if it is both a super- and a subsolution.
\vskip2pt
\noindent
By \cite[Theorem 1.1]{DCKP}, if $\Omega$ is bounded and $u\in W^{s,p}(\R^N)$ is a solution of \eqref{fpl}, satisfying a Dirichlet type condition in $\Omega^c$, then $u\in L^\infty_{\rm loc}(\Omega)$. We will assume such local boundedness in general.
\vskip2pt
\noindent
We will now recall some useful technical results that we are going to exploit in our arguments. We begin with a slightly modified version of \cite[Lemma 1.3]{DCKP}:

\begin{proposition}\label{log}
{\rm (logarithmic estimate)} Let $u\in\widetilde{W}^{s,p}_{\rm loc}(\Omega)$ be a supersolution of \eqref{fpl}, $\overline{B}_R(x_0)\subset\Omega$ s.t.\ $u\ge 0$ in $B_R(x_0)$, $d>0$, $\eta\in(0,1)$. Then, there exists $C_\eta>0$ depending on $N,p,s,\Omega,\Lambda_1,\Lambda_2$, and $\eta$ (with $C_\eta\to\infty$ as $\eta\to 1$) s.t.\ for all $r\in(0,\eta R)$
\[\iint_{B_r(x_0)\times B_r(x_0)}\Big|\log\Big(\frac{u(x)+d}{u(y)+d}\Big)\Big|^p\,\frac{dx\,dy}{|x-y|^{N+ps}} \le C_\eta r^{N-ps}\Big[1+\Big(\frac{r}{R}\Big)^{ps}\frac{{\rm Tail}(u_-,x_0,R)^{p-1}}{d^{p-1}}\Big].\]
\end{proposition}

\noindent
Also, from \cite[Theorem 1.4]{DCKP} we have:

\begin{proposition}\label{cac}
{\rm (Caccioppoli inequality)} Let $u\in\widetilde{W}^{s,p}_{\rm loc}(\Omega)$ be a solution of \eqref{fpl}, $\overline{B}_R(x_0)\subset\Omega$, $\in\R$, and set $w_\pm=(u-k)_\pm$. Then, there exists $C>0$ depending on $N,p,s,\Omega,\Lambda_1,\Lambda_2$, s.t.\ for all $\varphi\in C^1_c(B_R(x_0))$, $\varphi\ge 0$ in $B_R(x_0)$
\begin{align*}
&\iint_{B_R(x_0)\times B_R(x_0)}\frac{|w_\pm(x)\varphi(x)-w_\pm(y)\varphi(y)|^p}{|x-y|^{N+ps}}\,dx\,dy \\
&\le C\iint_{B_R(x_0)\times B_R(x_0)}\max\{w_\pm^p(x),w_\pm^p(y)\}\frac{|\varphi(x)-\varphi(y)|^p}{|x-y|^{N+ps}}\,dx\,dy \\
&+ C\Big[\int_{B_R(x_0)}w_\pm(x)\varphi^p(x)\,dx\Big]\,\Big[\sup_{y\in{\rm supp}(\varphi)}\,\int_{B_R^c(x_0)}\frac{w_\pm^{p-1}(x)}{|x-y|^{N+ps}}\,dx\Big].
\end{align*}
\end{proposition}

\begin{remark}\label{dir}
Note that in \cite{DCKP} both results above are stated for solutions of the Dirichet problem, but a careful reading of the proofs reveals that the Dirichlet condition is not actually exploited, and the results also hold for the free equation provided the nonlocal tails are finite, as is the case here.
\end{remark}

\noindent
Finally, we recall the aforementioned local clustering result. This can be deduced from \cite[Theorem 1.1]{DIV}, by replacing $N$-dimensional cubes with balls:

\begin{proposition}\label{lcl}
{\rm (local clustering)} Let $u\in W^{s,p}(B_R(x_0))$, $k>0$, $c_1\in(0,1)$, and $c_2>0$ satisfy
\begin{enumroman}
\item\label{lcl1} $\big|B_R(x_0)\cap\{u\ge k\}\big| \ge c_1|B_R(x_0)|$;
\item\label{lcl2} $[u]_{s,p,B_R(x_0)}\le c_2kR^\frac{N-ps}{p}$.
\end{enumroman}
Then for all $\lambda,\nu\in(0,1)$ there exist $\eps>0$ depending on $N,p,s,c_1,c_2$, and $x_1\in B_R(x_0)$ also depending on $u$, s.t.\ $B_{\eps R}(x_1)\subset B_R(x_0)$ and
\[\big|B_{\eps R}(x_1)\cap\{u\ge\lambda k\}\big| \ge (1-\nu)|B_{\eps R}(x_1)|.\]
\end{proposition}

\begin{remark}\label{alt}
The definition of solution presented above for equation \eqref{fpl} is not the only one in the current literature, for instance a less demanding one based on {\em tail spaces} was introduced in \cite{KKP}. Nevertheless, all these definitions turn out to be equivalent in the simplest cases, and with minor adjustments all the results above hold for our notion of solution. Analogously, we believe that the subsequent results of this paper can be adapted to different functional frameworks.
\end{remark}

\section{A De Giorgi type lemma}\label{sec3}

\noindent
Here we present a technical result, which ensures that whenever in a ball the relative measure of a sublevel set for a solution lies below a fixed threshold $\nu\in(0,1)$ (only depending on the data), then the measure of a lower sublevel set in the concentric ball with half radius is zero. Such result is sometimes labeled as {\em critical mass lemma}.
\vskip2pt
\noindent
We recall a simple algebraic lemma from \cite[Lemma 2.3]{DMV}:

\begin{lemma}\label{alg}
Let $(Y_j)$ be a sequence of positive numbers, $a>0$, $b>1$, $\beta>0$ s.t.\
\begin{enumroman}
\item\label{alg1} $Y_{j+1}\le ab^jY_j^{1+\beta}$ for all $j\in\N$;
\item\label{alg2} $Y_0\le a^{-\frac{1}{\beta}}b^{-\frac{1}{\beta^2}}$.
\end{enumroman}
Then, $Y_j\le Y_0b^{-\frac{j}{\beta}}$ for all $j\in\N$, in particular $Y_j\to 0$ as $j\to\infty$.
\end{lemma}

\noindent
Using Lemma \ref{alg} we prove the following:

\begin{proposition}\label{dgl}
{\rm (De Giorgi type lemma)} Let $u\in\widetilde{W}^{s,p}_{\rm loc}(\Omega)$ be a solution of \eqref{fpl}, $\overline{B}_{\bar\rho}(x_0)\subset\Omega$, $\rho_0\in(0,\bar\rho)$, $k>0$, $\mu\in\R$, $\gamma\in(0,1)$ be s.t.\
\[\mu \le \inf_{B_{\bar\rho}(x_0)}\,u, \ k > \gamma{\rm Tail}((u-\mu)_-,x_0,\rho_0).\]
Then, there exists $\nu\in(0,1)$ depending on $N,p,s,\Omega,\Lambda_1,\Lambda_2$, and $\gamma$, with the following property: if
\[\big|B_{\rho_0}(x_0)\cap\big\{u\le\mu+k\big\}\big| \le \nu|B_{\rho_0}(x_0)|,\]
then for a.e.\ $x\in B_{\rho_0/2}(x_0)$
\[u(x) > \mu+\frac{k}{2}.\]
\end{proposition}
\begin{proof}
Set for all $j\in\N$
\[\rho_j = \frac{1}{2}\Big(\rho_0+\frac{\rho_0}{2^j}\Big), \ \tilde\rho_j = \frac{\rho_j+\rho_{j+1}}{2}, \ k_j = \mu+\frac{1}{2}\Big(k+\frac{k}{2^j}\Big).\]
Clearly both $(\rho_j)$, $(k_j)$ are decreasing, with $\rho_j\to \rho_0/2$, $k_j\to\mu+k/2$ as $j\to\infty$. Set for all $j\in\N$
\[B_j = B_{\rho_j}(x_0), \ \tilde B_j = B_{\tilde\rho_j}(x_0), \ A_j = B_j\cap\{u\le k_j\}.\]
so $B_{j+1}\subset\tilde B_j\subset B_j$. Also set
\[w_j = (u-k_j)_-,\]
so $w_j\in\widetilde{W}^{s,p}_{\rm loc}(\Omega)$ is as well a solution of \eqref{fpl}, satisfying $w_j=|u-k_j|$ in $A_j$, $w_j=0$ in $B_j\setminus A_j$, $w_j\le k$ in all of $\R^N$, and for a.e.\ $x\in A_{j+1}$
\beq\label{dgl1}
w_j(x) = k_j-u(x) \ge k_j-k_{j+1}.
\eeq
Without loss of generality we may assume $\rho_0<1$. Fix $\varphi_j\in C^1_c(\tilde B_j)$ s.t.\ $\varphi_j=1$ in $B_{j+1}$, $0\le\varphi_j\le 1$ and $|D\varphi_j|\le C2^j/\rho_0$ in $\R^N$. Integrating \eqref{dgl1} over $A_{j+1}$ and applying H\"older's inequality, the fractional Sobolev's inequality, and Proposition \ref{cac}, we get
\begin{align}\label{dgl2}
|A_{j+1}| &\le \frac{1}{(k_j-k_{j+1})^p}\int_{A_{j+1}}w_j^p(x)\,dx \\
\nonumber &\le \frac{1}{(k_j-k_{j+1})^p}\int_{A_j}w_j^p(x)\varphi_j^p(x)\,dx \\
\nonumber &\le \frac{|A_j|^{1-\frac{p}{p^*_s}}}{(k_j-k_{j+1})^p}\Big[\int_{B_j}w_j^{p^*_s}(x)\varphi_j^{p^*_s}(x)\,dx\Big]^\frac{p}{p^*_s} \\
\nonumber &\le \frac{C|A_j|^{1-\frac{p}{p^*_s}}}{(k_j-k_{j+1})^p}\Big[\iint_{B_j\times B_j}\frac{|w_j(x)\varphi_j(x)-w_j(y)\varphi_j(y)|^p}{|x-y|^{N+ps}}\,dx\,dy+\int_{B_j}w_j^p(x)\varphi_j^p(x)\,dx\Big] \\
\nonumber &\le \frac{C|A_j|^{1-\frac{p}{p^*_s}}}{(k_j-k_{j+1})^p}\Big[\iint_{B_j\times B_j}\max\{w_j^p(x),w_j^p(y)\}\frac{|\varphi_j(x)-\varphi_j(y)|^p}{|x-y|^{N+ps}}\,dx\,dy\Big] \\
\nonumber &+ \frac{C|A_j|^{1-\frac{p}{p^*_s}}}{(k_j-k_{j+1})^p}\Big[\int_{B_j}w_j(x)\varphi_j^p(x)\,dx\Big]\,\Big[\sup_{y\in\tilde B_j}\,\int_{B_j^c}\frac{w_j^{p-1}(x)}{|x-y|^{N+ps}}\,dx\Big]+\frac{C|A_j|^{1-\frac{p}{p^*_s}}}{(k_j-k_{j+1})^p}\int_{A_j}w_j^p(x)\,dx \\
\nonumber &=: \frac{C|A_j|^{1-\frac{p}{p^*_s}}}{(k_j-k_{j+1})^p}\big[I_1+I_2]+\frac{C|A_j|^{1-\frac{p}{p^*_s}}}{(k_j-k_{j+1})^p}\int_{A_j}w_j^p(x)\,dx,
\end{align}
with $C>0$ depending on the data. We will now estimate separately $I_1$ and $I_2$. For $I_1$, we use symmetry, $w_j=0$ in $B_j\setminus A_j$, the bound on $|D\varphi_j|$, and the inequality $|x-y|\le 2\rho_j$ holding for all $x,y\in B_j$, to get
\begin{align*}
I_1 &\le C\iint_{A_j\times B_j}\max\{w_j^p(x),w_j^p(y)\}\frac{|\varphi_j(x)-\varphi_j(y)|^p}{|x-y|^{N+ps}}\,dx\,dy \\
&\le \frac{C2^{pj}k^p}{\rho_0^p}\iint_{A_j\times B_j}\frac{dx\,dy}{|x-y|^{N+ps-p}} \\
&\le \frac{C2^{pj}k^p|A_j|}{\rho_0^p}\sup_{y\in A_j}\int_{B_j}\frac{dx}{|x-y|^{N+ps-p}} \\
&\le \frac{C2^{pj}k^p|A_j|}{\rho_0^p}\int_{B_{2\rho_j}(0)}\frac{dz}{|z|^{N+ps-p}} \le \frac{C2^{pj}k^p|A_j|}{\rho_0^{ps}},
\end{align*}
where in the last line we have computed the integral on the ball $B_{2\rho_j}(0)$. We now deal with $I_2$, first noting that by the bounds on $w_j$, $\varphi_j$ we have
\[\int_{B_j}w_j(x)\varphi_j^p(x)\,dx \le k|A_j|.\]
Besides, recalling that for all $a,b\in\R$
\[(a-b)_+^{p-1} \le C_p(a_+^{p-1}+b_+^{p-1}),\]
for all $y\in\tilde B_j$ we have
\begin{align*}
\int_{B_j^c}\frac{w_j^{p-1}(x)}{|x-y|^{N+ps}}\,dx &\le C\int_{B_j^c}\frac{(k_j-\mu)_+^{p-1}}{|x-y|^{N+ps}}\,dx+C\int_{B_j^c}\frac{(\mu-u(x))_+^{p-1}}{|x-y|^{N+ps}}\,dx \\
&=: J_1+J_2.
\end{align*}
To estimate $J_1$, we note that for all $x\in B_j^c$, $y\in\tilde B_j$
\[|x-y| \ge \rho_j-\tilde\rho_j \ge \frac{\rho_0}{2^{j+3}},\]
hence uniformly for all $y\in\tilde B_j$ we have
\[J_1 \le Ck^{p-1}\int_{B_{\rho_j-\tilde\rho_j}^c(0)}\frac{dz}{|z|^{N+ps}} \le \frac{C2^{psj}k^{p-1}}{\rho_0^{ps}}.\]
To estimate $J_2$ we proceed similarly. First, for all $x\in B_j^c$, $y\in\tilde B_j$
\[\frac{|x-x_0|}{|x-y|} \le 1+\frac{|y-x_0|}{|x-y|} \le 2^{j+3},\]
hence for all $y\in\tilde B_j$, using the relation between $k$ and the tail, we have
\begin{align*}
J_2 &\le C2^{(N+ps)j}\int_{B_j^c}\frac{(u(x)-\mu)_-^{p-1}}{|x-x_0|^{N+ps}}\,dx \\
&\le \frac{C2^{(N+ps)j}}{\rho_0^{ps}}{\rm Tail}((u-\mu)_-,x_0,\rho_0)^{p-1} \le \frac{C2^{(N+ps)j}k^{p-1}}{\gamma^{p-1}\rho_0^{ps}}.
\end{align*}
The previous estimates and $\gamma\in(0,1)$ imply that for some $C>0$ depending on the data
\[I_2 \le \frac{C2^{(N+ps)j}k^p|A_j|}{\gamma^{p-1}\rho_0^{ps}}.\]
We go back to \eqref{dgl2}, recalling that $\rho_0<1$ and $k_j-k_{j+1}=k/2^{j+2}$:
\begin{align*}
|A_{j+1}| &\le \frac{C|A_j|^{1-\frac{p}{p^*_s}}}{(k_j-k_{j+1})^p}\,\Big[\frac{2^{pj}k^p|A_j|}{\rho_0^{ps}}+\frac{2^{(N+ps)j}k^p|A_j|}{\gamma^{p-1}\rho_0^{ps}}\Big]+\frac{C|A_j|^{1-\frac{p}{p^*_s}}}{(k_j-k_{j+1})^p}\int_{A_j}w_j^p(x)\,dx \\
&\le \frac{C2^{(N+p)j}k^p|A_j|^{2-\frac{p}{p^*_s}}}{\gamma^{p-1}(k_j-k_{j+1})^p\rho_0^{ps}}+\frac{Ck^p|A_j|^{2-\frac{p}{p^*_s}}}{(k_j-k_{j+1})^p} \\
&\le \frac{C2^{(N+2p)j}|A_j|^{1+\frac{ps}{N}}}{\gamma^{p-1}\rho_0^{ps}},
\end{align*}
where in the last line we have exploited the definition of $(k_j)$, and $C>0$ still depends on the data. We are going to apply Lemma \ref{alg} to the sequence
\[Y_j = \frac{|A_j|}{|B_j|}.\]
Dividing the inequality above by $|B_{j+1}|\sim C\rho_0^N$, we find the following recursive relation:
\begin{align*}
Y_{j+1} &= \frac{|A_{j+1}|}{|B_{j+1}|} \\
&\le \frac{C2^{(N+2p)j}|A_j|^{1+\frac{ps}{N}}}{\gamma^{p-1}\rho_0^{N+ps}} \le \frac{C2^{(N+2p)j}}{\gamma^{p-1}}Y_j^{1+\frac{ps}{N}}.
\end{align*}
The inequality above corresponds to hypothesis \ref{alg1} of Lemma \ref{alg} as soon as we set
\[a = \frac{C}{\gamma^{p-1}} > 0, \ b = 2^{N+2p} > 1, \ \beta = \frac{ps}{N} > 0.\]
Set now
\[\nu = a^{-\frac{1}{\beta}}b^{-\frac{1}{\beta^2}} = \Big(\frac{C}{\gamma^{p-1}}\Big)^{-\frac{N}{ps}} 2^{-\frac{N^2(N+2p)}{(ps)^2}},\]
so that $\nu\in(0,1)$ depends on the data and on $\gamma$. Assume now that
\[Y_0 = \frac{\big|B_{\rho_0}(x_0)\cap\big\{u\le\mu+k\big\}\big|}{|B_{\rho_0}(x_0)\big|} \le \nu.\]
Then hypothesis \ref{alg2} of Lemma \ref{alg} is satisfied as well. Thus, for all $j\in\N$ we have
\[Y_j \le Y_0b^{-\frac{j}{\beta}},\]
in particular $Y_j\to 0$ as $j\to\infty$. Passing to the limit, we see that
\[\Big|B_{\rho_0/2}(x_0)\cap\Big\{u \le \mu+\frac{k}{2}\Big\}\Big| = 0,\]
i.e.\ we have $u>\mu+k/2$ a.e.\ in $B_{\rho_0/2}(x_0)$, which concludes the proof.
\end{proof}

\section{First oscillation estimate}\label{sec4}

\noindent
In this section we obtain an oscillation estimate for the solutions of \eqref{fpl} in a ball contained in $\Omega$, by using a positivity expansion/recentering method:

\begin{proposition}\label{osc}
Let $u\in\widetilde{W}^{s,p}_{\rm loc}(\Omega)$ be a solution of \eqref{fpl}, $\overline{B}_R(x_0)\subset\Omega$. Then, there exist $\sigma>1$, $\theta\in(0,1)$ depending on $N,p,s,\Omega,\Lambda_1,\Lambda_2$ s.t.\
\[\underset{B_{\theta R/2}(x_0)}{\rm osc}\,u \le \Big(1-\frac{1}{32e^{\sigma}}\Big)\omega,\]
where
\[\omega = 2\sup_{B_R(x_0)}|u|+{\rm Tail}(u,x_0,R).\]
\end{proposition}

\noindent
The proof of Proposition \ref{osc} requires some preliminaries. First, for brevity we set $B_R=B_R(x_0)$ (similarly, the center $x_0$ will be omitted in the following, while we explicitly indicate different centers). We set
\[\mu_+ = \sup_{B_R}\,u, \ \mu_- = \inf_{B_R}\,u,\]
and hold on to the definition of $\omega$ given above. Clearly we have $\mu_+-\mu_-\le\omega$. We will henceforth assume
\beq\label{osc1}
\mu_+-\mu_- \ge \frac{\omega}{2},
\eeq
postponing the alternative case to the end. Under \eqref{osc1}, at least one of the following holds: either
\beq\label{osc2}
\Big|B_{R/4}\cap\Big\{u\ge\mu_-+\frac{\omega}{4}\Big\}\Big| \ge \frac{|B_{R/4}|}{2},
\eeq
or
\beq\label{osc3}
\Big|B_{R/4}\cap\Big\{u\le\mu_+-\frac{\omega}{4}\Big\}\Big| \ge \frac{|B_{R/4}|}{2},
\eeq
Indeed, if \eqref{osc2} fails, then clearly
\[\Big|B_{R/4}\cap\Big\{u\le\mu_-+\frac{\omega}{4}\Big\}\Big| \ge \frac{|B_{R/4}|}{2},\]
and since by \eqref{osc1}
\[\mu_-+\frac{\omega}{4} = \mu_++(\mu_--\mu_+)+\frac{\omega}{4} \le \mu_+-\frac{\omega}{4},\]
we see that \eqref{osc3} holds. We will assume \eqref{osc2} in the following lemmas. First, we prove a lower bound for $u$ in a small ball contained in $B_{R/4}$, with a shifted center $x_1$:

\begin{lemma}\label{clu}
{\rm (positivity clustering)} Let \eqref{osc1}, \eqref{osc2} hold. Then, there exist $\eps\in(0,1)$ depending on $N,p,s,\Omega,\Lambda_1,\Lambda_2$ and $x_1\in B_{R/4}$ also depending on $u$, s.t.\ $B_{\eps R/8}(x_1)\subset B_{R/4}$ and for a.e.\ $x\in B_{\eps R/8}(x_1)$
\[u(x) > \mu_-+\frac{\omega}{16}.\]
\end{lemma}
\begin{proof}
Set $w=u-\mu_-\in\widetilde{W}^{s,p}_{\rm loc}(\Omega)$, still a solution of \eqref{fpl}. We have in $B_R$
\[0 \le w \le \mu_+-\mu_- \le \omega.\]
We aim at applying Proposition \ref{lcl} to the function $w$. First, \eqref{osc2} rephrases as
\[\Big|B_{R/4}\cap\Big\{w\ge\frac{\omega}{4}\Big\}\Big| \ge \frac{|B_{R/4}|}{2}.\]
Now fix $\varphi\in C^1_c(B_{3R/8})$ s.t.\ $\varphi=1$ in $B_{R/4}$, and $0\le\varphi\le 1$, $|D\varphi|\le C/R$ in all of $\R^N$. By Proposition \ref{cac} we have
\begin{align}\label{clu1}
\iint_{B_{R/4}\times B_{R/4}}\frac{|w(x)-w(y)|^p}{|x-y|^{N+ps}}\,dx\,dy &\le \iint_{B_{R/2}\times B_{R/2}}\frac{|w(x)\varphi(x)-w(y)\varphi(y)|^p}{|x-y|^{N+ps}}\,dx\,dy \\
\nonumber&\le C\iint_{B_{R/2}\times B_{R/2}}\max\{w^p(x),w^p(y)\}\frac{|\varphi(x)-\varphi(y)|^p}{|x-y|^{N+ps}}\,dx\,dy \\
\nonumber&+ C\Big[\int_{B_{R/2}}w(x)\varphi^p(x)\,dx\Big]\,\Big[\sup_{y\in B_{3R/8}}\,\int_{B_{R/2}^c}\frac{w_+^{p-1}(x)}{|x-y|^{N+ps}}\,dx\Big] \\
\nonumber&=: I_1+I_2.
\end{align}
We separately estimate $I_1$ and $I_2$, with a similar argument as in Proposition \ref{dgl}. For $I_1$, recall that $w\le\omega$, $|D\varphi|\le C/R$ in $B_{R/2}$:
\begin{align*}
I_1 &\le \frac{C\omega^p}{R^p}\iint_{B_{R/2}\times B_{R/2}}\frac{dx\,dy}{|x-y|^{N+ps-p}} \\
&\le \frac{C\omega^p}{R^p}\sup_{y\in B_{R/2}}\,|B_{R/2}|\int_{B_{R/2}}\frac{dx}{|x-y|^{N+ps-p}} \\
&\le C\omega^pR^{N-p}\int_{B_R(0)}\frac{dz}{|z|^{N+ps-p}} \le C\omega^p R^{N-ps}.
\end{align*}
Now we focus on $I_2$. First, we have
\[\int_{B_{R/2}}w(x)\varphi^p(x)\,dx \le C\omega R^N.\]
Next, note that for all $x\in B_{R/2}^c$, $y\in B_{3R/8}$
\[\frac{|x-x_0|}{|x-y|} \le 1+\frac{|y-x_0|}{|x-y|} \le 4.\]
So, recalling that $u\le\mu_+$ in $B_R\setminus B_{R/2}$, we have uniformly for all $y\in B_{3R/8}$
\begin{align*}
\int_{B_{R/2}^c}\frac{w_+^{p-1}(x)}{|x-y|^{N+ps}}\,dx &\le C\int_{B_{R/2}^c}\frac{(u(x)-\mu_-)_+^{p-1}}{|x-x_0|^{N+ps}}\,dx \\
&\le C\int_{B_R^c}\frac{(u(x)-\mu_+)_+^{p-1}}{|x-x_0|^{N+ps}}\,dx+C\int_{B_{R/2}^c}\frac{(\mu_+-\mu_-)^{p-1}}{|x-x_0|^{N+ps}}\,dx \\
&\le \frac{C}{R^{ps}}{\rm Tail}((u-\mu_+)_+,x_0,R)^{p-1}+\frac{C\omega^{p-1}}{R^{ps}}.
\end{align*}
From the definition of $\omega$ we have $|\mu_+|\le\omega$, hence
\begin{align}\label{clu2}
{\rm Tail}((u-\mu_+)_+,x_0,R)^{p-1} &\le CR^{ps}\int_{B_R^c}\frac{|u(x)|^{p-1}}{|x-x_0|^{N+ps}}\,dx+CR^{ps}\int_{B_R^c}\frac{|\mu_+|^{p-1}}{|x-x_0|^{N+ps}}\,dx \\
\nonumber &\le C{\rm Tail}(u,x_0,R)^{p-1}+C\omega^{p-1}R^{ps}\int_{B_R^c}\frac{dx}{|x-x_0|^{N+ps}} \le C\omega^{p-1}.
\end{align}
Plugging \eqref{clu2} into the previous inequality we get
\[I_2 \le C\omega^pR^{N-ps}.\]
Using the estimates of $I_1$, $I_2$ in \eqref{clu1} and raising to the power $1/p$, we have
\[[w]_{s,p,B_{R/4}} \le C\omega R^\frac{N-ps}{p}.\]
Now we see that hypotheses \ref{lcl1} and \ref{lcl2} of Proposition \ref{lcl} hold. Fix $\lambda=1/2$, $\nu\in(0,1)$ to be determined later. Then, we can find $\eps\in(0,1)$ depending on the data and $\nu$, and $x_1\in B_{R/4}$ (also depending on $u$) s.t.\ $B_{\eps R/4}(x_1)\subset B_{R/4}$ and
\beq\label{clu3}
\Big|B_{\eps R/4}(x_1)\cap\Big\{u\ge\mu_-+\frac{\omega}{8}\Big\}\Big| \ge (1-\nu)|B_{\eps R/4}(x_1)|.
\eeq
Now we need to convert the measure-theoretical information \eqref{clu3} into a lower bound on $u$, possibly reducing both the level and the radius. This step is made possible by Proposition \ref{dgl}, which we apply to the function $u$ with $\rho_0=\eps R/4$, $\mu=\mu_-$ (satisfying $u\ge\mu_-$ in $B_R$), and $k=\omega/8$. We claim that there exists $\gamma\in(0,1)$ depending on the data, s.t.\
\beq\label{clu4}
\frac{\omega}{8} \ge \gamma{\rm Tail}\Big((u-\mu_-)_-,x_1,\frac{\eps R}{4}\Big).
\eeq
Indeed, for all $x\in B_R^c$ we have
\[\frac{|x-x_1|}{|x-x_0|} \ge 1-\frac{|x_1-x_0|}{|x-x_0|} \ge 1-\Big(1-\frac{\eps}{4}\Big) = \frac{\eps}{4}.\]
Also, $u\ge\mu_-$ in $B_R$, hence
\begin{align*}
{\rm Tail}\Big((u-\mu_-)_-,x_1,\frac{\eps R}{4}\Big)^{p-1} &= \Big(\frac{\eps R}{4}(x_1)\Big)^{ps}\int_{B_{\eps R/4}^c}\frac{(u(x)-\mu_-)_-^{p-1}}{|x-x_1|^{N+ps}}\,dx \\
&\le CR^{ps}\int_{B_R^c}\frac{(u(x)-\mu_-)_-^{p-1}}{|x-x_0|^{N+ps}}\,dx \le C\omega^{p-1},
\end{align*}
where the last inequality is obtained as in \eqref{clu2} and we may assume $C>1$ without loss of generality. So we may set $\gamma=1/8C^\frac{1}{p-1}\in(0,1)$ to get \eqref{clu4}.
\vskip2pt
\noindent
Let $\nu\in(0,1)$, depending on the data (note that $\gamma$ depends in turn on the data), be the number given by Proposition \ref{dgl}, and choose such $\nu$ in \eqref{clu3}. Reversing the inequality, we have
\[\Big|B_{\eps R/4}(x_1)\cap\Big\{u\le\mu_-+\frac{\omega}{8}\Big\}\Big| \le \nu|B_{\eps R/4}(x_1)|.\]
By Proposition \ref{dgl} we have for a.e.\ $x\in B_{\eps R/8}(x_1)$
\[u(x) > \mu_-+\frac{\omega}{16},\]
which concludes the proof.
\end{proof}

\noindent
The bound of Lemma \ref{clu} is not yet enough for our purposes, mainly because the center $x_1$ of the ball, where the bound holds, is unknown {\em a priori}. So, we will now prove a (looser) bound in an even smaller ball, but centered at $x_0$:

\begin{lemma}\label{rec}
{\em (positivity recentering)} Let \eqref{osc1}, \eqref{osc2} hold. Then, there exist $\sigma>1$, $\theta\in(0,1)$ depending on $N,p,s,\Omega,\Lambda_1,\Lambda_2$ s.t.\ for a.e.\ $x\in B_{\theta R/2}$
\[u(x) > \mu_-+\frac{\omega}{32e^\sigma}.\]
\end{lemma}
\begin{proof}
Fix $\sigma>\log(3)$, $\theta\in(0,1/4)$ to be chosen later. We have
\[\log\Big(\frac{e^\sigma+1}{2}\Big) \ge \sigma-1.\]
Set $r=\theta R\in(0,R/4)$, and let $\eps\in(0,1)$, $x_1\in B_{R/4}$ be given by Lemma \ref{clu}, so that for a.e.\ $x\in B_{\eps r/2}(x_1)$ we have
\[u(x) > \mu_-+\frac{\omega}{16}.\]
Define the sublevel set
\[A_{\sigma,r} = B_r\cap\Big\{u\le\mu_-+\frac{\omega}{16e^\sigma}\Big\}.\]
For a.e.\ $x\in B_{\eps r/2}(x_1)$, $y\in A_{\sigma,r}$ we have
\begin{align*}
\frac{(u(x)-\mu_-)+\omega/16e^\sigma}{(u(y)-\mu_-)+\omega/16e^\sigma} &\ge \frac{\omega/16+\omega/16e^\sigma}{\omega/8e^\sigma} \\
&= \frac{e^\sigma+1}{2} > 1,
\end{align*}
hence
\[\log\Big[\frac{(u(x)-\mu_-)+\omega/16e^\sigma}{(u(y)-\mu_-)+\omega/16e^\sigma}\Big] \ge \sigma-1.\]
Next we integrate in both $x$, $y$, and apply Proposition \ref{log} to the function $u-\mu_-$ (still a solution of \eqref{fpl}) with $d=\omega/16e^\sigma>0$, $\eta=\theta\in(0,1)$:
\begin{align*}
(\sigma-1)^p|B_{\eps r/2}(x_1)||A_{\sigma,r}| &\le \iint_{B_{\eps r/2}(x_1)\times A_{\sigma,r}}\Big|\log\Big[\frac{(u(x)-\mu_-)+\omega/16e^\sigma}{(u(y)-\mu_-)+\omega/16e^\sigma}\Big]\Big|^p\,dx\,dy \\
&\le Cr^{N+ps}\iint_{B_r\times B_r}\Big|\log\Big[\frac{(u(x)-\mu_-)+\omega/16e^\sigma}{(u(y)-\mu_-)+\omega/16e^\sigma}\Big]\Big|^p\,\frac{dx\,dy}{|x-y|^{N+ps}} \\
&\le Cr^{2N}\Big[1+\Big(\frac{r}{R}\Big)^{ps}\Big(\frac{16e^\sigma}{\omega}\Big)^{p-1}{\rm Tail}((u-\mu_-)_-,x_0,R)^{p-1}\Big].
\end{align*}
Note that $C>0$ depends in principle on $\theta$, but since $\theta<1/4$ we may replace it with a bigger constant, only depending on the data. Arguing as in \eqref{clu2} we find $\gamma\in(0,1)$ depending on the data s.t.\
\beq\label{rec1}
\omega > \gamma{\rm Tail}((u-\mu_-)_-,x_0,R).
\eeq
Dividing by $(\sigma-1)^p$ and by $|B_r|,|B_{\eps r/2}(x_1)|\sim Cr^N$, and using \eqref{rec1}, we get
\begin{align*}
\frac{|A_{\sigma,r}|}{|B_r|} &\le \frac{Cr^{2N}}{(\sigma-1)^p |B_r||B_{\eps r/2}(x_1)|}\Big[1+\Big(\frac{r}{R}\Big)^{ps}\Big(\frac{16e^\sigma}{\gamma}\Big)^{p-1}\Big] \\
&\le \frac{C}{(\sigma-1)^p}\big[1+e^{\sigma(p-1)}\theta^{ps}\big].
\end{align*}
Now we impose on $\theta$ and $\sigma$ a first constraint:
\beq\label{rec2}
\theta < \Big(\frac{\gamma}{16e^\sigma}\Big)^\frac{p-1}{ps}.
\eeq
Again we need to apply Proposition \ref{dgl} to $u$, this time with $\rho_0=r$, $\mu=\mu_-$, and $k=\omega/16e^\sigma$. Indeed, by \eqref{rec1} and \eqref{rec2}, and recalling that $u\ge\mu_-$ in $B_R$, we have
\begin{align*}
\frac{\omega}{16e^\sigma} &> \frac{\gamma{\rm Tail}((u-\mu_-)_-,x_0,R)}{16e^\sigma} \\
&= \frac{\gamma}{16e^\sigma}\Big(\frac{R}{r}\Big)^\frac{ps}{p-1}{\rm Tail}((u-\mu_-)_-,x_0,r) \\
&> {\rm Tail}((u-\mu_-)_-,x_0,r).
\end{align*}
Let $\nu\in(0,1)$ be given by Proposition \ref{dgl}, depending on the data (note that, according to the computation above, we may take $\gamma=1$). We can find $\sigma>\log(3)$, depending as well on the data, s.t.\
\[\frac{C}{(\sigma-1)^p} < \frac{\nu}{2}.\]
By taking $\theta$ even smaller if necessary (but still depending on the data) we ensure \eqref{rec2}, which in turn implies
\[1+e^{\sigma(p-1)}\theta^{ps} < 1+\frac{\gamma^{p-1}}{16^{p-1}} < 2.\]
So, from the previous relations we get
\[|A_{\sigma,r}| \le \nu|B_r|.\]
By Proposition \ref{dgl}, we have for a.e.\ $x\in B_{\theta R/2}$
\[u(x) > \mu_-+\frac{\omega}{32e^\sigma},\]
which concludes the proof.
\end{proof}

\noindent
We can now complete the proof of our oscillation estimate, considering all the cases:
\vskip4pt
\noindent
{\em Proof of Proposition \ref{osc}.} First assume \eqref{osc1}. If \eqref{osc2} holds as well, we can apply Lemma \ref{rec} and find $\sigma>1$, $\theta\in(0,1)$ depending on the data, s.t.\ $u>\mu_-+\omega/32e^\sigma$ in $B_{\theta R/2}$. This in turn implies
\[\underset{B_{\theta R/2}}{\rm osc}\,u \le \mu_+-\Big(\mu_-+\frac{\omega}{32e^\sigma}\Big) \le \Big(1-\frac{1}{32e^\sigma}\Big)\omega.\]
Now let \eqref{osc1} hold, and \eqref{osc2} fail. Then, as observed above, we have \eqref{osc3}. We argue exactly as in Lemmas \ref{clu} and \ref{rec}, replacing $u$ with $-u$ (a function with the same oscillation), $\mu_+$ with $-\mu_-$, $\mu_-$ with $-\mu_+$, while keeping the same $\omega$. So we find $\sigma>1$, $\theta\in(0,1)$ depending on the data, s.t.\  for a.e.\ $x\in B_{\theta R/2}$
\[-u(x) > -\mu_++\frac{\omega}{32e^\sigma}.\]
Therefore we have
\begin{align*}
\underset{B_{\theta R/2}}{\rm osc}\,u &= \underset{B_{\theta R/2}}{\rm osc}\,(-u) \\
&\le -\mu_--\Big(-\mu_++\frac{\omega}{32e^\sigma}\Big) \le \Big(1-\frac{1}{32e^\sigma}\Big)\omega.
\end{align*}
Finally, assume that \eqref{osc1} fails. Then, for arbitrary $\sigma>1$, $\theta\in(0,1)$ we have
\[\underset{B_{\theta R/2}}{\rm osc}\,u \le \mu_+-\mu_- \le \frac{\omega}{2} < \Big(1-\frac{1}{32e^\sigma}\Big)\omega.\]
Thus, the conclusion is achieved in all cases. \qed

\section{Iterative oscillation estimate}\label{sec5}

\noindent
In this section we improve Proposition \ref{osc} to an iterative oscillation estimate with the same definition of $\omega$, by means of a strong induction argument (this is a typical trick in fractional regularity results, as it allows to manage tail terms):

\begin{proposition}\label{ito}
Let $u\in\widetilde{W}^{s,p}_{\rm loc}(\Omega)$ be a solution of \eqref{fpl}, $\overline{B}_R(x_0)\subset\Omega$. Then, there exist $\sigma>1$, $\theta\in(0,1)$ depending on $N,p,s,\Omega,\Lambda_1,\Lambda_2$ s.t.\ for all $n\in\N$
\[\underset{B_{(\theta/2)^nR}(x_0)}{\rm osc}\,u \le \Big(1-\frac{1}{32e^\sigma}\Big)^n\omega,\]
where
\[\omega = 2\sup_{B_R(x_0)}\,|u|+{\rm Tail}(u,x_0,R).\]
\end{proposition}

\noindent
The proof of Proposition \ref{ito} follows the same path as that of Proposition \ref{osc}, but with some important differences that need to be explicitly pointed out. Again we omit the center of balls as long as it is $x_0$, and define $\mu_+$, $\mu_-$ as in Section \ref{sec4}. For all $n\in\N$ we set
\[R_n = \Big(\frac{\theta}{2}\Big)^nR, \ B_n = B_{R_n}, \ \mu_+^n = \sup_{B_n}\,u, \ \mu_-^n = \inf_{B_n}\,u, \ \omega_n = \Big(1-\frac{1}{32e^\sigma}\Big)^n\omega,\]
with $\sigma>1$, $\theta\in(0,1)$ yet to be determined. From Proposition \ref{osc} we know that, for convenient $\sigma$, $\theta$ depending on the data, we have
\[\underset{B_1}{\rm osc}\,u \le \omega_1.\]
In order to prove the same assertion for all $n\in\N$, we argue by strong induction. Assume that for some $\sigma>1$, $\theta\in(0,1)$ depending on the data we have
\beq\label{ito1}
\underset{B_i}{\rm osc}\,u \le \omega_i \ (i=1,\ldots n).
\eeq
In the following we will show that an analogous estimate holds at step $(n+1)$. As in Section \ref{sec4} we first assume
\beq\label{ito2}
\mu_+^n-\mu_-^n \ge \frac{\omega_n}{2}.
\eeq
Like above, this leaves us with at least one of the following true: either
\beq\label{ito3}
\Big|B_{R_n/4}\cap\Big\{u\ge\mu_-^n+\frac{\omega_n}{4}\Big\}\Big| \ge \frac{|B_{R_n/4}|}{2},
\eeq
or
\beq\label{ito4}
\Big|B_{R_n/4}\cap\Big\{u\le\mu_+^n-\frac{\omega_n}{4}\Big\}\Big| \ge \frac{|B_{R_n/4}|}{2}.
\eeq
Assuming \eqref{ito3}, we first prove an iterative version of Lemma \ref{clu}:

\begin{lemma}\label{itc}
{\rm (iterative positivity clustering)} Let \eqref{ito1}, \eqref{ito2}, \eqref{ito3} hold. Then, there exist $\eps\in(0,1)$ depending on $N,p,s,\Omega,\Lambda_1,\Lambda_2$ and $x_1\in B_{R_n/4}$ also depending on $u$, s.t.\ $B_{\eps R_n/8}(x_1)\subset B_{R_n/4}$ and for a.e.\ $x\in B_{\eps R_n/8}(x_1)$
\[u(x) > \mu_-^n+\frac{\omega_n}{16}.\]
\end{lemma}
\begin{proof}
The argument follows, up to a point, that of Lemma \ref{clu}. We set $w_n=u-\mu_-^n\in\widetilde{W}^{s,p}_{\rm loc}(\Omega)$, still a solution of \eqref{fpl} s.t.\ by \eqref{ito1} we have in $B_n$
\[0 \le w_n \le \omega_n.\]
Besides, by \eqref{ito3} we have
\beq\label{itc1}
\Big|B_{R_n/4}\cap\Big\{w_n\ge\frac{\omega_n}{4}\Big\}\Big| \ge \frac{|B_{R_n/4}|}{2}.
\eeq
Fix $\varphi_n\in C^1_c(B_{3R_n/8})$ s.t.\ $\varphi_n=1$ in $B_{R_n/4}$, $0\le\varphi_n\le 1$, $|D\varphi_n|\le C/R_n$ in $\R^N$. Arguing as in Lemma \ref{clu} we have
\begin{align}\label{itc2}
[w_n]_{s,p,B_{R_n/4}}^p &\le \iint_{B_{R_n/2}\times B_{R_n/2}}\frac{|w_n(x)\varphi_n(x)-w_n(y)\varphi_n(y)|^p}{|x-y|^{N+ps}}\,dx\,dy \ \ldots \\
\nonumber&\le C\omega_n^pR_n^{N-ps}+C\omega_nR_n^{N-ps}{\rm Tail}((u-\mu_+^n)_+,x_0,R_n)^{p-1}.
\end{align}
We need a new estimate of the tail term, which can be obtained via the following decomposition:
\begin{align*}
{\rm Tail}((u-\mu_+^n)_+,x_0,R_n)^{p-1} &= R_n^{ps}\int_{B_0^c}\frac{(u(x)-\mu_+^n)_+^{p-1}}{|x-x_0|^{N+ps}}\,dx+\sum_{i=1}^nR_n^{ps}\int_{B_{i-1}\setminus B_i}\frac{(u(x)-\mu_+^n)_+^{p-1}}{|x-x_0|^{N+ps}}\,dx \\
&=: J_0+\sum_{i=1}^n J_i.
\end{align*}
We estimate $J_0$ as in \eqref{clu2}, noting that $|\mu_+^n|\le\omega$:
\begin{align*}
J_0 &\le CR_n^{ps}\int_{B_0^c}\frac{|u(x)|^{p-1}}{|x-x_0|^{N+ps}}\,dx+CR_n^{ps}\int_{B_0^c}\frac{|\mu_+^n|^{p-1}}{|x-x_0|^{N+ps}}\,dx \\
&\le C\Big(\frac{R_n}{R}\Big)^{ps}{\rm Tail}(u,x_0,R)^{p-1}+C\omega^{p-1}R_n^{ps}\int_{B_0^c}\frac{dx}{|x-x_0|^{N+ps}} \\
&\le C\Big(\frac{\theta}{2}\Big)^{psn}\omega^{p-1}+\frac{C\omega^{p-1}R_n^{ps}}{R^{ps}} \\
&\le C\Big(\frac{\theta}{2}\Big)^{psn}\Big(1-\frac{1}{32e^\sigma}\Big)^{-(p-1)n}\omega_n^{p-1}.
\end{align*}
Now we pick $i\in\{1,\ldots n\}$ and estimate $J_i$. Recall that, by the inductive hypothesis \eqref{ito1}, we have for a.e.\ $x\in B_{i-1}\setminus B_i$
\[(u(x)-\mu_+^n)_+ \le (\mu_+^{i-1}-\mu_-^n)_+ \le \mu_+^{i-1}-\mu_-^{i-1} \le \omega_{i-1}\]
So we have
\begin{align*}
J_i &\le C\omega_{i-1}^{p-1}R_n^{ps}\int_{B_{i-1}\setminus B_i}\frac{dx}{|x-x_0|^{N+ps}} \\
&\le C\omega_{i-1}^{p-1}\Big(\frac{R_n}{R_i}\Big)^{ps} \\
&\le C\Big(\frac{\theta}{2}\Big)^{ps(n-i)}\Big(1-\frac{1}{32e^\sigma}\Big)^{(p-1)(i-1-n)}\omega_n^{p-1}.
\end{align*}
Adding the previous inequalities, we get
\[{\rm Tail}((u-\mu_+^n)_+,x_0,R_n)^{p-1} \le C\omega_n^{p-1}\sum_{i=0}^n\Big[\Big(\frac{\theta}{2}\Big)^{ps}\Big(1-\frac{1}{32e^\sigma}\Big)^{1-p}\Big]^{n-i}.\]
We can choose $\theta\in(0,1)$, depending on the data, so small that
\[\Big(\frac{\theta}{2}\Big)^{ps}\Big(1-\frac{1}{32e^\sigma}\Big)^{1-p} \le \frac{1}{2}.\]
Then, from the previous inequality we deduce
\[{\rm Tail}((u-\mu_+^n)_+,x_0,R_n)^{p-1} \le C\omega_n^{p-1}\sum_{i=0}^\infty \frac{1}{2^i} = 2C\omega_n^{p-1}.\]
Plugging such estimate into \eqref{itc2} we get
\beq\label{itc3}
[w_n]_{s,p,B_{R_n/4}} \le C\omega_nR_n^\frac{N-ps}{p}.
\eeq
By \eqref{itc1}, \eqref{itc3} we can apply Proposition \ref{lcl}, again with $\lambda=1/2$ and $\nu\in(0,1)$ to be chosen later. So there exist $\eps\in(0,1)$ depending on the data, $x_1\in B_{R_n/4}$ also depending on $u$, s.t.\ $B_{\eps R_n/4}(x_1)\subset B_{R_n/4}$ and
\beq\label{itc4}
\Big|B_{\eps R_n/4}(x_1)\cap\Big\{u\ge\mu_-^n+\frac{\omega_n}{8}\Big\}\Big| \ge (1-\nu)|B_{\eps R_n/4}(x_1)|.
\eeq
The next step consists in applying Proposition \ref{dgl} to $u$ with $\rho_0=\eps R_n/4$, $\mu=\mu_-^n$, and $k=\omega_n/8$. Note that there exists $\gamma\in(0,1)$ depending on the data s.t.\
\[\frac{\omega_n}{8} > \gamma{\rm Tail}\Big((u-\mu_-^n)_-,x_1,\frac{\eps R_n}{4}\Big).\]
Indeed, as in Lemma \ref{clu}, recalling that $u\ge\mu_-^n$ in $B_n$, we have
\begin{align*}
{\rm Tail}\Big((u-\mu_-^n)_-,x_1,\frac{\eps R_n}{4}\Big)^{p-1} &= \Big(\frac{\eps R_n}{4}\Big)^{ps}\int_{B_n^c}\frac{(u(x)-\mu_-^n)_-^{p-1}}{|x-x_1|^{N+ps}}\,dx \\
&\le C{\rm Tail}((u-\mu_-^n)_-,x_0,R_n)^{p-1} \le C\omega_n^{p-1},
\end{align*}
where in the last line we have used a decomposition as above and we may assume $C>1$. So let $\gamma=1/8C^{1/(p-1)}\in(0,1)$ to prove the desired inequality. Now let $\nu\in(0,1)$, depending on the data, be as in Proposition \ref{dgl}, and choose such $\nu$ in \eqref{itc4} to find
\[\Big|B_{\eps R_n/4}\cap\Big\{u\le\mu^n_-+\frac{\omega_n}{8}\Big\}\Big| \le \nu|B_{\eps R_n/4}|.\]
By Proposition \ref{dgl}, we have for a.e.\ $x\in B_{\eps R_n/8}(x_1)$
\[u(x) > \mu_-^n+\frac{\omega_n}{16},\]
which concludes the proof.
\end{proof}

\noindent
Next we recenter our ball at step $n$:

\begin{lemma}\label{itr}
{\rm (iterative positivity recentering)} Let \eqref{ito1}, \eqref{ito2}, \eqref{ito3} hold. Then, there exist $\sigma>1$, $\theta\in(0,1)$ depending on $N,p,s,\Omega,\Lambda_1,\Lambda_2$ s.t.\ for a.e.\ $x\in B_{n+1}$
\[u(x) > \mu_-^n+\frac{\omega_n}{32e^\sigma}.\]
\end{lemma}
\begin{proof}
Again, the argument partially follows that of Lemma \ref{rec}. Fix $\sigma>\log(3)$, $\theta\in(0,1/4)$ to be chosen later. As in Lemma \ref{rec}, applying Lemma \ref{itc} we find
\[\frac{\big|B_{\theta R_n}\cap\{u\le\mu_-^n+\omega_n/16e^\sigma\}\big|}{|B_{\theta R_n}|} \le \frac{C}{(\sigma-1)^p}\Big[1+\theta^{ps}\Big(\frac{16e^\sigma}{\omega_n}\Big)^{p-1}{\rm Tail}((u-\mu_-^n)_-,x_0, R_n)^{p-1}\Big].\]
As in Lemma \ref{itc} we find $\gamma\in(0,1)$ depending on the data s.t.\
\beq\label{itr1}
\omega_n \ge \gamma{\rm Tail}((u-\mu_-^n)_-,x_0,R_n),
\eeq
which along with the previous inequality gives
\beq\label{itr2}
\frac{\big|B_{\theta R_n}\cap\{u\le\mu_-^n+\omega_n/16e^\sigma\}\big|}{|B_{\theta R_n}|} \le \frac{C}{(\sigma-1)^p}\Big[1+\theta^{ps}\Big(\frac{16e^\sigma}{\gamma}\Big)^{p-1}\Big].
\eeq
Again we require that $\theta$ satisfy \eqref{rec2}, which we recall for simplicity:
\beq\label{itr3}
\theta < \Big(\frac{\gamma}{16e^\sigma}\Big)^\frac{p-1}{ps}.
\eeq
Using \eqref{itr1}, \eqref{itr3}, and recalling that $u\ge\mu_-^n$ in $B_n$, we have
\begin{align*}
\frac{\omega_n}{16e^\sigma} &\ge \frac{\gamma}{16e^\sigma}{\rm Tail}((u-\mu_-^n)_-,x_0,R_n) \\
&= \frac{\gamma}{16e^\sigma}R_n^\frac{ps}{p-1}\Big[\int_{B_n^c}\frac{(u(x)-\mu_-^n)_-^{p-1}}{|x-x_0|^{N+ps}}\,dx\Big]^\frac{1}{p-1} \\
&> (\theta R_n)^\frac{ps}{p-1}\Big[\int_{B_{\theta R_n}^c}\frac{(u(x)-\mu_-^n)_-^{p-1}}{|x-x_0|^{N+ps}}\,dx\Big]^\frac{1}{p-1} \\
&= {\rm Tail}((u-\mu_-^n)_-,x_0,\theta R_n).
\end{align*}
We can now apply Proposition \ref{dgl} to $u$, with $\rho_0=\theta R_n$, $\mu=\mu_-^n$, $k=\omega_n/16e^\sigma$. Let $\nu\in(0,1)$ be given by Proposition \ref{dgl}. We can find $\sigma>\log(3)$ depending on the data, s.t.\
\[\frac{C}{(\sigma-1)^p} < \frac{\nu}{2}.\]
Recalling that $\theta\in(0,1/4)$ satisfies \eqref{itr3}, we also get
\[1+\theta^{ps}\Big(\frac{16e^\sigma}{\gamma}\Big)^{p-1} < 2.\]
Therefore, \eqref{itr2} implies
\[\Big|B_{\theta R_n}\cap\Big\{u\le\mu^n_-+\frac{\omega_n}{16e^\sigma}\Big\}\Big| \le \nu|B_{\theta R_n}|.\]
Now Proposition \ref{dgl} implies that for a.e.\ $x\in B_{\theta R_n/2}=B_{n+1}$
\[u(x) > \mu_-^n+\frac{\omega_n}{32e^\sigma},\]
which concludes the proof.
\end{proof}

\noindent
We can now complete the proof of the iterative oscillation estimate:
\vskip4pt
\noindent
{\em Proof of Proposition \ref{ito}.} For $n=1$, the conclusion follows from Proposition \ref{osc}. So, we fix $n\in\N$, assume \eqref{ito1}, and prove that the conclusion holds at step $n+1$ (strong induction), distinguishing several cases. If \eqref{ito2} holds, then either \eqref{ito3} or \eqref{ito4} holds. Under \eqref{ito3}, Lemma \ref{itr} ensures the existence of $\sigma>1$, $\theta\in(0,1)$ depending on the data s.t.\ $u\ge\mu_-^n+\omega_n/32e^\sigma$ in $B_{n+1}$. So we have
\[\underset{B_{n+1}}{\rm osc}\,u \le \mu_+^n-\Big(\mu_-^n+\frac{\omega_n}{32e^\sigma}\Big) \le \omega_{n+1}.\]
If \eqref{ito3} fails, then \eqref{ito4} holds. Then we argue as in Lemmas \ref{itc}, \ref{itr} above, replacing $u$ with $-u$, and we find that in $B_{n+1}$
\[-u > -\mu_+^n+\frac{\omega_n}{32e^\sigma},\]
hence,
\begin{align*}
\underset{B_{n+1}}{\rm osc}\,u &=  \underset{B_{n+1}}{\rm osc}\,(-u) \\
&\le -\mu_-^n-\Big(-\mu_+^n+\frac{\omega_n}{32e^\sigma}\Big) \le \omega_{n+1}.
\end{align*}
Finally, if \eqref{ito2} fails, then trivially
\[\underset{B_{n+1}}{\rm osc}\,u \le \underset{B_n}{\rm osc}\,u < \frac{\omega_n}{2} < \omega_{n+1}.\]
In all case, by strong induction we have that for all $n\in\N$
\[\underset{B_n}{\rm osc}\,u \le \omega_n,\]
which concludes the proof. \qed

\section{Conclusion: proof of H\"older continuity}\label{sec6}

\noindent
In this final section we prove our main result. We use a classical method, as in \cite{DMV}, yet we prefer to include a detailed proof for the reader's convenience:
\vskip2pt
\noindent
{\em Proof of Theorem \ref{hol}.} Fix a ball $\overline{B}_R(x_0)\subset\Omega$. By Proposition \ref{ito}, there exist real numbers $0<a<b<1$ depending on the data, s.t.\ for all $n\in\N$
\[\underset{B_{a^nR}}{\rm osc}\, u \le b^n\omega,\]
where we omit the center $x_0$ and as in Proposition \ref{ito} we have set
\[\omega = 2\sup_{B_R}\,|u|+{\rm Tail}(u,x_0,R) > 0.\]
Since $\log(a)<\log(b)<0$, we may define
\[\alpha := \frac{\log(b)}{\log(a)} \in (0,1),\]
depending on the data, so that $b=a^\alpha$. We claim that for all $r\in(0,R)$
\beq\label{hol1}
\underset{B_r}{\rm osc}\,u \le \frac{\omega}{b}\Big(\frac{r}{R}\Big)^\alpha.
\eeq
Indeed, we can find $n\in\N$ s.t.\ $a^nR\le r<a^{n-1}R$. Then, by the definition of $\alpha$ we have
\[b^{n-1} \le \frac{b^{n-1}}{a^{\alpha n}}\Big(\frac{r}{R}\Big)^\alpha = \frac{1}{b}\Big(\frac{r}{R}\Big)^\alpha,\]
hence by the iterative estimate
\[\underset{B_r}{\rm osc}\,u \le \underset{B_{a^{n-1}R}}{\rm osc}\,u \le b^{n-1}\omega \le \frac{\omega}{b}\Big(\frac{r}{R}\Big)^\alpha.\]
The oscillation estimate \eqref{hol1} suffices to prove local (H\"older) continuity of $u$. To be more precise, let $\Omega'\Subset\Omega$ s.t.\ ${\rm dist}(\Omega',\Omega^c) = 2R > 0$. Define the $R$-dilatation of $\Omega'$ as the compact set
\[\Omega'_R = \big\{x\in\Omega:\,{\rm dist}(x,\Omega')\le R\big\}.\]
Since $u$ is locally bounded, we may set
\[\overline\omega = \sup_{x\in\Omega'_R}\,\big[2|u(x)|+{\rm Tail}(u,x,R)\big] \in (0,\infty).\]
We claim that there exists $C>0$ depending on the data, s.t.\ for all $x_1,x_2\in\Omega'$
\beq\label{hol2}
|u(x_1)-u(x_2)| \le C\overline\omega\Big(\frac{|x_1-x_2|}{R}\Big)^\alpha
\eeq
(recall that $u$ is continuous, hence the relation above holds for {\em all} $x_1$, $x_2$). We set $r=|x_1-x_2|>0$, and distinguish two cases:
\begin{itemize}[leftmargin=1cm]
\item[$(a)$] If $r<R/2$, then $x_1,x_2\in\overline{B}_r(x_1)$ and by \eqref{hol1} we have
\begin{align*}
|u(x_1)-u(x_2)| &\le \underset{B_r(x_1)}{\rm osc}\,u \\
&\le \frac{1}{b}\Big[2\sup_{B_R(x_1)}\,|u|+{\rm Tail}(u,x_1,R)\Big]\Big(\frac{r}{R}\Big)^\alpha \le C\overline\omega\Big(\frac{|x_1-x_2|}{R}\Big)^\alpha.
\end{align*}
\item[$(b)$] If $r\ge R/2$, then we can find $y_1,\ldots y_m\in\Omega$ ($m\ge 1$ depending on $\Omega'$) s.t.\ $x_1 \in B_{R/2}(y_1)$, $x_2 \in B_{R/2}(y_m)$, $\overline{B}_{R/2}(y_i) \subset \Omega$ for all $i\in\{1,\ldots m\}$, and $|y_i-y_{i+1}|<R/2$ for all $i\in\{i,\ldots m-1\}$. So, by \eqref{hol1} we have
\begin{align*}
|u(x_1)-u(x_2)| &\le |u(x_1)-u(y_1)|+\sum_{i=1}^{m-1}|u(y_i)-u(y_{i+1})|+|u(y_m)-u(x_2)| \\
&\le \underset{B_{R/2}(y_1)}{\rm osc}\,u+\sum_{i=1}^{m-1}\underset{B_{R/2}(y_i)}{\rm osc}\,u+\underset{B_{R/2}(y_m)}{\rm osc}\,u \\
&\le C\overline\omega\Big(\frac{1}{2}\Big)^\alpha \le C\overline\omega\Big(\frac{r}{R}\Big)^\alpha.
\end{align*}
\end{itemize}
In both cases we have \eqref{hol2}, which concludes the proof. \qed

\vskip4pt
\noindent
{\bf Acknowledgement.} The authors are members of GNAMPA (Gruppo Nazionale per l'Analisi Matematica, la Probabilit\`a e le loro Applicazioni) of INdAM (Istituto Nazionale di Alta Matematica 'Francesco Severi'). A.I.\ is partially supported by the research project {\em Problemi non locali di tipo stazionario ed evolutivo} (GNAMPA, CUP E53C23001670001). We would like to thank the anonymous Referees for their useful comments.

\end{document}